\documentclass[12pt,reqno]{article}
\usepackage{graphicx}
\usepackage{amsmath}
\usepackage{amsthm}
\usepackage{latexsym,bm}
\usepackage{amssymb}
\usepackage{indentfirst}
\newtheorem{thm}{Theorem}[section]

\newtheorem{lem}[thm]{Lemma}
\newtheorem{prop}[thm]{Proposition}
\newtheorem{defn}[thm]{Definition}
\newtheorem{rem}[thm]{Remark}
\newtheorem{exm}{Example}

\newtheorem{ques}[thm]{Question}
\numberwithin{equation}{section}

\begin{document}

\title{\bf{\Large Weak embedding theorem and a proof of cycle double cover of bridgeless graphs}}

\author{{\bf Bin Shen} }

\date{}

\maketitle

\begin{quote}
\small {\bf Abstract}.  In this article, 
we give a positive answer to the cycle double cover conjecture. Ones who are mainly interesting in the proof of the conjecture can only read Sections 2 and 4.
\end{quote}
\begin{quote}
\small {\bf Mathematics Subject Classification}: 05C38 05C45\\
\small {\bf Keywords}: cycle double cover conjecture; bridgeless graph; embedding theorem.
\end{quote}

\baselineskip 17pt

\section{Introduction}

Cycle double cover is an unsolved problem in graph-theoretic mathematics, which was raised by W.T. Tutte\cite{Tutte}, G. Szekeres\cite{Szek} and P.D. Seymour\cite{Sey}.
The conjecture queries whether every bridgeless undirected graph has a cycle double cover, that is, a collection of cycles so that each edge of the graph is contained in exactly two of the cycles.

As is well known, the most interesting results is the observation of F. Jaeger\cite{Jae}. In any potential minimal counterexample of the cycle double cover conjecture,
all vertices must have three or more incident edges. Moreover, if a vertex $v$ has four or more incident edges,
one may ``split off" two of those edges by removing them from the graph and replacing them with a single edge connecting their other two endpoints,
while preserving the bridgelessness of the resulting graph. On the other hand,
a double cover of the resulting graph may be extended in a straightforward way to a double cover of the original graph. More precisely,
every circle of the split off graph corresponds either to a cycle of the original graph, or to a pair of cycles meeting at $v$. Thus,
every minimal counterexample must be cubic. So they defined a \emph{snark} to be a bridgeless graph,
with the additional properties that every vertex has exactly three incident edges and that it is not possible to partition the edges of the graph into three perfect
matchings. The former property implies that the graph is cubic, while the latter means that the graph has no 3-edge-colorable bridgeless cubic graph.

Another important approach to the conjecture is to embed such a bridgeless graph into an oriented two dimensional topological manifold, namely, a surface.
One may hope that the boundary curves of faces which covers this manifold containing the graph might give the double cover,
since the target topological manifold is oriented. But it fails because some of the boundaries are not cycles of the graph.
It implies that the cycle assumption is more strict than the boundaries of a topological division on the oriented manifold. On the other hand, cycles also can not be used as
the boundaries of some pieces to divide the target manifold. This fact is explained in Section 3.
\\

In this paper, we first discuss the weak embedding theorem,
which claims that every bridgeless graph can be embedded into a two dimensional compact oriented topological manifold with a precise Euler characteristic.
After that, we provide a positive answer to the cycle double cover conjecture of a bridgeless graph.
\begin{thm}\label{main theorem}
Every bridgeless graph admits a cycle double cover.
\end{thm}
The details are given in Section 4.

\section{Some concepts}
We call an undirected graph short for graph and denote it by $G$. Vertices of $G$ are given by $v_i$ for some $i$.
Any edge between two vertices $v_i, v_j$ is denoted by $e_{ij}$, or sometimes by $v_iv_j$. The number of incident edges of a vertex $v_i$ is called \emph{degree} and is denoted by $d_i$.
A vertex $v_{\alpha}$ is called an \emph{odd vertex} if its degree $d_{\alpha}$ is odd, and is called an \emph{even vertex} if $d_{\alpha}$ is an even nonzero integer.

A graph $G$ is said to be \emph{complete}, if every two vertices in $G$ have a edge. A complete graph is totally determined by the number of its vertices.
Therefore, a complete graph is always denoted by $K(n)$ if it has $n$ vertices.

Moreover, there are several terminologies need to explain.

\emph{Walk}: A walk in a graph can pass some edges more than once from a vertex to a vertex.

\emph{Closed walk}: A closed walk is a walk which is from a vertex to itself.

\emph{Trail}: A trail is a walk that can not pass any edge twice, but can pass some vertices more than once.

\emph{Circuit}: A circuit is a closed trail, namely, with the same starting and ending vertex.

\emph{Path}: A path is a trail which can not pass any vertex twice.

\emph{Cycle}: A cycle is a closed path.

\emph{Fork}: A fork is a special closed walk that pass any edge exactly twice. In a fork, the vertices passed once are called the \emph{ending vertices}, and the other vertices are called the \emph{inner vertices}. The vertices passed more than twice are called \emph{bifurcation vertices}.

\emph{Segment}: A segment is a special fork that pass each vertex no more than twice.

A segment is one-to-one corresponding to a path, since it can be represented by
$$s=v_1v_2\cdots v_{k-1}v_kv_{k-1}\cdots v_2v_1.$$
The vertices $v_1,v_k$ are called the \emph{ending vertices} of the segment $s$.

A closed walk may pass some edges more than twice. So a walk $w$ is a combination of its edges, and some edges may repeatedly appear in the combination. We denote the edge-set of a walk $w$ by $e(w)$, and denote the induced graph of $e(w)$ by $G[e(w)]$. We call $G[e(w)]$ the \emph{induced graph of $w$}. Furthermore, if $W$ is a set of walks, we define the edge-set of $W$, which is denoted by $e(W)$, is the minimal set that contains $e(w)$ for every $w\in W$. The \emph{induced graph of $W$} is defined to ge the induced graph of $e(w)$, and is denoted by $G[e(W)]$.

Given a set $A$ of circuits (or segments, or forks, respectively), we can connect two elements in the set together to get a larger circuits (or segments, or forks, respectively), and replacing the original two elements by the new one. Continuing this process to get a new set $B$ of circuits (or segments, or forks, respectively), such that any two element in $B$ can not be connected to form a larger new circuits (or segments, or forks, respectively). Then we call the element in $B$ is \emph{maximal}. We call a circuit or a segment is \emph{irreducible}, if the circuit or a segment does not contain any other such structure. More precisely, an irreducible circuit is a cycle. An irreducible segment is a segment expressed by $v_1v_2v_1$, namely, whose corresponding path has length one. We call a subgraph or a subset is \emph{proper}, if it is not the empty or the graph or the set itself.

The above terminologies are all used to describe graphs. However, one may consider such a (closed) walk (a circuit, a segment, a fork, etc.) as a sequence of vertices, which give the corresponding graph a special order. This order implies a possible way to pass each edge. It could not cause any ambiguity, if we focus on the graphs and the covering sheets (one or two-sheets) of each edge. For example, sometimes we say that ``a closed walk consists of some cycles and forks". It means that there is a decomposition of the induced graph of the walk into some cycles and trees (which is corresponding to the forks), such that it passes each edge of the graphs of cycles once and it pass each edge of the trees twice. When we say ``Connecting two segments together", the phrase means to connect the corresponding paths of two segments together and adopting a walk passing each edge exactly twice.

When we mention the ``containing", we are discussing in the level of ``set of edges". For example, a circuit $r_2$ containing another circuit $r_1$ means that the edges of $r_1$ is a subset of the edges of $r_2$. A segment $s$ being contained in a fork $f$ means that the edges of $s$ is a subset of the edges of $f$, and each edge of $s$ is exactly passed twice in $f$. A fork $f$ containing a cycle $r$ means that the edges of $r$ is a subset of the fork $f$. Each edge of $r$ is passed once in $r$, while twice in $f$.

\section{Weak embedding theorem}

One may wonder if we shall embed a given bridgeless graph into a manifold. If so, we can cut the manifold into pieces along the edges of the graph.
Then the boundaries of these pieces can provide a double cover of the graph. But the geometric plan can not solve the conjecture.
We will first prove the weak embedding theorem and then give a counterexample to illustrate the failure of this method. However,
it gives us some important thoughts. Firstly, we deduce the following conclusion.
\begin{lem}
Any complete graph can be embedded into a two-dimensional oriented compact topological manifold.
\end{lem}
\begin{proof}
The proof is quite obvious. Let us see some simple examples with less vertices first.

If the complete graph has 4 vertices, the number of edges is 6. It can be embedded into a sphere, by considering the graph as a tetrahedron.

If the complete graph has 5 vertices. The number of edges is 10. It can not be embedded into a sphere, but into a torus.

Similarly, we can use the genuses to separate the new edges. As a result,
we adopt the induction on the number of vertices and use multi-connected parts to locate the new edges,
which is the incident edges of the $n-$th vertex to the $(n-1)-$th complete graph. The lemma has been obtained.
\end{proof}

Now we want to give the precise genuses, or Euler characteristics of the target topological manifolds. That is

\begin{thm}\label{thm1}
Any complete graph $K(k)$, with $k\geq3$, can be embedded into a two dimensional oriented compact topological manifold with genus of $\frac{(k-3)(k-4)}{2}$,
or with Euler characteristic of $2-(k-3)(k-4)$. Hence any graph with $k$ vertices must can be embedded into such a manifold.
\end{thm}
\begin{proof}
If $k=3$, the graph can be embedded into a plane hence into a sphere. When $k=4$, the graph also can be embedded into a sphere as we have shown above.

Now we assume the result holds for any complete graph with $k-$vertices, that is,
a complete graph of $k$ vertices can be embedded into a two dimensional oriented compact topological manifold with Euler characteristic of $2-(k-3)(k-4)$.
Then a complete graph with $k+1$ vertices can be obtained by adding a new vertex $v_{k+1}$ to a complete graph of $k-$vertices and
completing all the incident edges of $v_{k+1}$. Therefore,
one more vertex combines with $k$ more edges. Not all of these $k$ additional edges need to be separated by multi-connected parts.
Three of them can be embedded into the manifold of Euler characteristic of $2-(k-3)(k-4)$. This is because there are three vertices combined with
the new vertex $v_{k+1}$ and the edges between them build a simplex, which does not change the topology of the manifold when it is pasted to the complete graph of $k-$vertices.
So we only need $k-3$ more genus to get a new oriented compact manifold. Namely its Euler characteristic is
$$\chi_{k+1}=2-2(\frac{(k-3)(k-4)}{2}+(k-3))=2-(k-3)(k-2).$$
\end{proof}

From Theorem \ref{thm1}, we see that any graph can be embedded into a two dimensional manifold with sufficient large genus, whose upper bound is determined by the number of vertices of the graph. So we give the following definition.
\begin{defn}
We call a nonnegative integer $k$ the genus of a graph $G$, if $G$ can be embedded into a two dimensional oriented manifold of genus $k$ and can not be embedded
into a a two dimensional oriented manifold of genus $k-1$.
\end{defn}

There is a natural question that
\begin{ques}
How to determine the genus of an arbitrary graph?
\end{ques}
This question is equal to how to find the equivalent complete graph of any given bridgeless graph. If we classify all graphs to different classes by different genuses,
the question is equal to find the representation of each class. However, this problem is N-P hard.

Noticing a graph isomorphism is a map $f$ between two graphs $G_1$ and $G_2$ satisfying
\begin{itemize}
  \item (1) $f$ sends every vertex of $G_1$ to a vertex of $G_2$ and $f$ is a bijective;
  \item (2) $f$ sends every edge of $G_1$ to an edge of $G_2$ and $f$ is a bijective.
\end{itemize}
We define a weak sense of isomorphism between graphs.
\begin{defn}
Two graphs $G_1$ and $G_2$ are called homotopic to each other if $G_1$ and $G_2$ have the same genus.
\end{defn}

A more intrinsic question is that
\begin{ques}
What is the graphic characteristic of a graph and the topological correspondence of its corresponding embedded manifold?\\
\end{ques}

Now we illustrate the complete graph $K(5)$ with 5 vertices to point out the failure of the embedding method to solve the cycle double cover conjecture.

\begin{exm}
The failure of embedding method in the case of $K(5)$.
\end{exm}
$K(5)$ can be embedded into a torus as shown in Figure \ref{K5}.
\begin{figure}[h!]
\begin{center}
  \includegraphics[width=12cm]{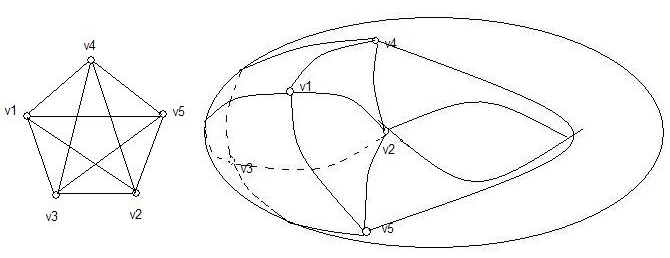}\\
  \caption{K(5) in a torus}\label{K5}
  \end{center}
\end{figure}
Cutting the torus along each edges of $K(5)$ provides pieces with boundaries
$$v_4v_3v_1,\quad v_4v_2v_1,\quad v_3v_1v_5,\quad v_1v_2v_5,\quad v_3v_2v_4v_5v_2v_3v_5v_4.$$
The boundary of the last piece is not a cycle, since $v_4v_5$ and $v_2v_3$ have been passed twice relatively.

One the other hand, there is an obvious cycle double covering of $K(5)$, by noticing that $K(5)$ is an eulerian graph. That is,
$$v_4v_1v_3v_2v_5v_4\quad \mbox{and}\quad v_4v_3v_5v_1v_2v_4.$$

As is well known, a torus $T$ is equal to $T=S^1\times S^1$. These two cycles go around the two $S^1$'s once respectively.
If we cut the torus along any one of these two cycles, the surface we get is homeomorphic to the manifold which is obtained by cutting
the M\"obius strip along the axle circle. That is,
$$T_2\setminus\{S_1,S_2\}\simeq M\setminus S_3,$$
where $M$ denotes the M\"obius strip, $T_2$ denotes the standard torus, $S_1=v_4v_1v_3v_2v_5v_4$, $S_2=v_4v_3v_5v_1v_2v_4$ and the center cycle of a M\"obius strip $S_3$ are all isomorphic to standard $S^1$.

Hence the cycles cannot be considered as the boundaries of some simply connected oriented pieces.
\qed\\

\section{Two-sheets covering of a graph and the proof of cycle double cover conjecture}

Although the embedding method can not solve the conjecture as we have illustrated above, it still provides us a way to investigate this question.

The important fact is that any two dimensional compact oriented manifold has some two-sheets covering. Let's see some examples. The two-sheets covering of
a sphere is double spheres, which are not connected. There are four kinds of two-sheets coverings of a torus. One is the double toruses and any one of the others is isomorphic to a torus.
So does the other manifold with more genuses.

Now suppose we have embedded a graph into a compact oriented manifold. If we lift the graph into one of the connected two-sheets covering spaces
(which could be achieved except the case of sphere), then we will get a two-sheets covering of the graph.
Considering that a graph which could be embedded into a sphere must be an eulerian graph, it is trivial and can be omitted.
Since the two-sheets covering of the graph is located on a connected manifold, these independent graphs can be connected in some way.
If so, the conjecture of cycle double cover becomes an eulerian path problem of the ``connected" two-sheets covering graph.

This program can be simplified into a description of graph theory. Now we use this idea to solve the problem.
\\

\textbf{Proof of Theorem\ref{main theorem}:} We use $G$ to denote a graph $G(v_i,e_{ij})$, with $\{e_{ij}=v_iv_j\}$ denoting the edges and $\{v_i\}$ representing the vertices of $G$. We usually use $v_iv_j$ to denote the edge from $v_i$ to $v_j$ with length 1, and $v_iv_2\cdots v_k$ to denote a path from $v_1$ to $v_k$ with length $k$.
\\

\emph{Step 1:} Lifting the graph $G$ to a connected two-sheets covering space, we get two isomorphic graphs $G_1(v_i^1,e_{ij}^1)$ and $G_2(v_i^2,e_{ij}^2)$.
We can simply consider $G_1$ and $G_2$ as two usual copies of graph $G$. There are two kinds of vertices in $G$, called the odd and even vertices, according to their degrees.
We pick up all the odd vertices $v_{\alpha}$ in $G$ with $d_{\alpha}\equiv1$(mod 2), and connect $v_{\alpha}^1$ and $v_{\alpha}^2$ for all couples of odd vertices to get a new connected graph $\tilde{G}$.
By denoting the \emph{auxiliary edge} which connects the pair of odd vertices $v_i^1$ and $v_i^2$ by $c_i=v^1_iv^2_i$,
the new graph $\tilde{G}$ can be described as $\tilde{G}=\tilde{G}(v_i^1,v_k^2;e_{ij}^1,e_{kl}^2,c_{\alpha})=\tilde{G}(v_{\xi};e_{\xi\eta},c_{\alpha})$.
We can see easily that any vertex in $\tilde{G}$ has an even degree, namely,  $d_{\xi}\equiv0$(mod 2) for any $v_{\xi}\in\tilde{G}$.
There is an extreme case that $G$ is a graph which only contains even vertices, that is, $d_i\equiv0$(mod 2) for any $v_i\in G$.
If so, we can obtain $\tilde{G}$ by just connecting one arbitrary pair of vertices $v_j^1$ and $v_j^2$, e.g. connecting $v_1^1$ and $v_1^2$. Thereby,
the auxiliary edge $c_1$ is a bridge of $\tilde{G}$. We find that all the vertices in $\tilde{G}$ have even degrees, except $v_1^1$ and $v_1^2$.
So in either cases, we get an eularian graph $\tilde{G}$.
Then $\tilde{G}$ can be covered by an eulerian path, which is a circuit in the former case. Denote the eulerian path by $\tilde E$.

We define the projection map from $\tilde G$ to $G$ by $p$ with $p(v^{\iota}_i)=v_i$, $p(e^{\iota}_{ij})=e_{ij}$ and $p(c_i)=0$.

We denote each cycle in $\tilde E$ that does not pass any auxiliary edge $c_i$ by $\tilde l_p$. Such a cycle is completely contained in $G_1$ or $G_2$. We pick them out to form a set of cycles
\begin{eqnarray*}
\tilde L=\{\tilde l_p=v^a_{p_1}v^a_{p_2}\cdots v^a_{p_k}v^a_{p_1}|p_1\cdots p_k\in\{1,2,\cdots,n\},p_1\neq\cdots \neq p_k,k\geq3, n=1 \,\mbox{or}\,2\}.
\end{eqnarray*}
The projection of each $\tilde l_p$ is still a cycle in $G$. All the projections form a projective set of $\tilde L$, denoted by $L$. Each element $l$ in $L$ is a projection of one element $\tilde l$ in $\tilde L$ (two $l$'s may exactly have the same edges). $\tilde l$ is called the \emph{one-sheet lifting of $l$}. $\tilde L$ is called the \emph{one-sheet lifting of $L$}.

$\tilde{G}$ contains a bridge if and only if $G$ is an eulerian graph. The reason is that $G$ is bridgeless and the odd vertices must appear in couple. Hence all the cycles in $\tilde E$ must be completely contained in $G_1$ or $G_2$, if all the vertices in $G$ are even vertices. So $\tilde L$ is the cycle double cover of $G$ in this case.

From now on, we only need to consider the case that $G$ contains some odd vertices.\\

\emph{Step 2:} When $G$ has odd vertices, $\tilde E-\tilde L$ is a union of circuits in $\tilde G$, which must cover all the set of auxiliary edges $C=\{c_i\}$. $\tilde E-\tilde L$ could be decomposed into a sequence of cycles, which form a set $\tilde M=\{\tilde m_i\}$ by
\begin{eqnarray*}
\begin{split}
\tilde M&=\{\tilde m_i=a_{i_1}a_{i_2}\cdots a_{i_k}a_{i_1}|a_{i_j}\in\{v^1_{l}\}\cup\{v^2_k\}\}\\
&=\{\tilde m_i=b_{i_1i_2}b_{i_2i_3}\cdots b_{i_ki_1}|b_{i_ji_l}\in\{e^1_{pq}\}\cup\{e^2_{st}\}\cup C\}.
\end{split}
\end{eqnarray*}
Considering the closeness of each $m_i$, such a cycle $m_i$ must pass even numbers of auxiliary edges, that is, $|\tilde m_i\cap C|$ is even.

The projection of each $\tilde m_i$ is denoted by $m_i=p(\tilde m_i)$, which is a closed walk in $G$. All such projections form a set $M=\{m_i\}$. For an $m_i\in M$, the edges it passed can be decomposed into two sets. The first one consists of all the edges passed only once by $m_i$, while the second set contains the remaining edges passed exactly twice. If we consider the covering sheet of $m_i$ on each edge, the two sets can be regarded as a set of cycles in $m_i$ and a set of segments in $m_i$, respectively. We denote them by $Q^i_1$ and $Q^i_2$, respectively. Moreover, we can assume each element in $Q^i_i$ or $Q^i_2$ is irreducible. Namely, $m_i=Q^i_1\cup Q^i_2$ with
\begin{eqnarray*}
\begin{split}
Q^i_1&=\{q_j=v_{j_1}v_{j_2}\cdots v_{j_k}v_{j_1}|v_{j_r}\neq v_{j_s}, \forall r\neq s,k\geq 3\},\\
Q^i_2&=\{q_j=v_{j_1}v_{j_2}v_{j_1}\}.
\end{split}
\end{eqnarray*}
We set $Q_1=\cup_i Q^i_1$ and $Q_2=\cup_iQ^i_2$, where the summation is taken over all $m_i\in M$.

Now we connect the cycles in $Q_1$ as long as possible to get a new set of circuits. Precisely, suppose $q_1,q_2\in Q_1$ are two cycles in $Q_1$. If $q_1, q_2$ only have common vertices, we can get a larger circuit. If $q_1, q_2$ have common edges, which must form some segments in $q_1\cup q_2$, we can get a larger circuit by taking off the segments. Inductively, we get a set $R$ of all such maximal circuits in $Q_1$. All the segments we have taken off form a set $S_T$.
\begin{rem}
$R$ is not unique, which depends on the order of $q_i\in Q$ to connect together. However, circuits in $R$ are all maximal.
\end{rem}

We connect the segments in $Q^i_2$ as long as possible to get a set $S^i_C$. Precisely, suppose $s_1,s_2\in S^i_C$ are two segments, which share a common ending vertex, then we connect them together to get a get a larger segment. Inductively, we get a set, still denoted by $S^i_C$. An element in $S^i_C$ is a segment, whose induced graph is a path. For each $m_i\in M$, the union of sets $S^i_C$ provides a set $\{S^i_C\}=\cup_iS^i_C$.

We call a closed walk a \emph{double cycle}, if the induced graph of it is a cycle, and each edge of the cycle is passed exactly twice by the closed walk. Connecting the segments in $\{S^i_C\}$ as long as possible, one can get a set of segments $S_C$ and a set of double cycles $D_C$. Precisely, any two segments $s_1,s_2\in\{S^i_C\}$ sharing one or two common ending vertices are connected together at the common ending vertices to form a larger segment or double circuit. It is obvious that such segments are contained in different elements in $M$, i.e., $s_1\in M_{i_1}$, $s_2\in m_{i_2}$ and $i_1\neq i_2$. Inductively, we get a set of segments $S_C$ and a set of double set $D_C$.

We union the two sets $S_T$ and $S_C$, and connect the segments in the union set as long as possible to get a set of segments $S$ and a set of double circuits $D_T$, by repeating the above process on elements in $S_T$ and $S_C$. Let $D=D_C\cup D_T$. We get a set of maximal segments $S$ and a set of double circuits $D$. Any element $s$ in $S$ is maximal because any two segments in $S$ do not share any common ending vertex.
\begin{rem}
$S$ is not unique. However, segments in $S$ are all maximal.
\end{rem}

A segment $s$ in $G$ has two kinds of ending vertices:
\begin{itemize}
  \item[(I):] odd ending vertex.
  \item[(II):] even ending vertex.
\end{itemize}
\begin{prop}
The second kind of ending vertex (II) can not appear in a segment $s\in S$.
\end{prop}
\begin{proof}
Suppose $s\in S$ is a maximal segment in $S$. Denote its two ending vertices by $v_1$ and $v_k$. If $v_1$ is an even vertex, without loss of generality, we can assume that the degree of $v_1$ is at least 3. Hence, besides the segment $s$, there must be another segment or cycle passing through $v_1$, either of which contributes 2 degrees to $v_1$. Hence there is another segment starting from $v_1$, which is a contradiction to the maximal of $s\in S$.
\end{proof}
Denote the projection of the eulerian path $\tilde E$ in $G$ by $E$, which is a closed walk. $E$ can be decomposed into three parts by
\begin{eqnarray}\label{ELRS}
E=L \cup R\cup D\cup S,
\end{eqnarray}
where $\cup$ means the union. $L$ is a set of cycles. $R$ is a set of circuits, which also can be regarded as a set of cycles. $D$ is set of double circuits, which also can be considered as a set of cycles. We define an element $d$ in $D$ to be a cycle, instead of a double cycle for uniformity. $S$ is a set of maximal segments.

Since $D$ is a set of double circuits, a cycle $d$ in $D$ must one-to-one corresponding to a cycle in $\tilde G$, denoted by $\tilde d$. We call $\tilde d$ is the o\emph{ne-sheet lifting of $d$}. The union of $\tilde d$ form a set $\tilde D$. $\tilde D$ is called the \emph{one-sheet lifting of $D$}.

For any maximal segment $s$ in $S$, we can lift $s$ into $\tilde G$ in the following way. Let $s=v_{i_1}\cdots v_{i_{k-1}}v_{i_k}v_{i_{k-1}}\cdots v_{i_1}$ be a maximal segment in $S$, where $v_{i_1}$ and $v_{i_k}$ are both odd vertices. Any vertex $v_{i_j}$ are lifted into $v_{i_j}^1$ and $v_{i_j}^2$ to form two paths in $G_1$ and $G_2$, respectively. Then we connect $v^1_{i_1}$ and $v^2_{i_1}$ by an auxiliary edge $c_{i_1}$, and connect $v^1_{i_k}$ and $v^2_{i_k}$ by an auxiliary edge $c_{i_k}$. Such auxiliary edges exist for the odd vertices $v_{i_1}$ and $v_{i_k}$. In this way, any segment $s$ in $S$ can be lifted into a cycle $\tilde s$ in $\tilde E$, which is called a lifting segment. $\tilde s$ is called the \emph{two-sheets lifting of $s$}. All the lifting segments form a set $\tilde S$. $\tilde S$ is called the \emph{two-sheets lifting of $S$}. It is easy to see that $p(\tilde s)=s$, for any $\tilde s\in \tilde S$.

$R$ is a set of maximal circuits, which also can be regarded as a set of cycles. We use the flexible terminology ``circuit" to describe an element in $R$. An element in $R$ is a circuit, which may not be maximal or irreducible. All the elements in $R$ form a one-sheet covering of the edge-set of $R$. We consider the set of walks $\tilde R=\tilde E-\tilde L-\tilde D-\tilde S$. It is obvious that $\tilde R$ can be considered as a set of disconnected circuits since elements in $\tilde E$, $\tilde L$, $\tilde D$ and $\tilde S$ are all circuits. More precisely, it satisfies that
\begin{prop}\label{prop4.4}
Consider elements in $\tilde R$ to be the disconnected circuits in $\tilde E$. The elements in $R$ are regarded as circuits to be determined. $\tilde R$ is a one-sheet lifting of $R$, that is, every element $\tilde r$ (a maximal circuit in $\tilde R$) of $\tilde R$ is a one-sheet covering of a circuit $r$ in $R$. Furthermore, for all elements in $\tilde R$, their projections form a one-sheet covering of the edge-set of $R$.
\end{prop}
\begin{proof}
$\tilde R$ is a set of disconnected circuits. Such a circuit can be considered topologically as a connected component of $\tilde R$. $\tilde E$ is an eulerian path, hence is a circuit when there is some odd vertices in $G$. 
Noticing the degree of any vertex in a circuit is even, every connected component $R_C$ of $\tilde R$ still admits an eulerian path, because all the degrees of vertices in $\tilde E$ are even. We claim that there is no cut edge in $\tilde R$. Otherwise, taking this cut edge away in a connected component $\tilde R_C$ splits $\tilde R_C$ into two disconnected parts, each of which contains only one odd vertex, since all the vertices in $\tilde R_C$ are even. But that is impossible, since odd vertices must appear in pairs. So each connected component in $\tilde R$ is a closed walk, hence is a maximal circuit.

Denote the projection of a connected component $\tilde R_C$ in $\tilde R$ by $R_C$. Since $\tilde R_C$ is a circuit in $\tilde G$, $p(\tilde R_C)$ is a closed walk in $G$. From the construction of $S$, there is no any segment in $p(\tilde R_C)$, so every element of $\tilde R$ is a one-sheet covering of a circuit $r$ in $R$. We can rearrange the set $R$ to be the union of projections of disconnected circuits in $\tilde R$. Because
$$p(\tilde R)=p(\tilde E-\tilde L-\tilde D-\tilde S)=p(\tilde E)-p(\tilde L)-p(\tilde D)-p(\tilde S)=E-L-D-S=R,$$
We finish the proof.
\end{proof}

Now we get
\begin{eqnarray*}
\tilde E=\tilde L\dot\cup\tilde R\dot\cup\tilde D\dot\cup\tilde S,
\end{eqnarray*}
where $\dot\cup$ means the disjoint union, $\tilde L$ is the one-sheet lifting of $L$, $\tilde R$ is the one-sheet lifting of $R$, $\tilde D$ is the one-sheet lifting of $D$, and $\tilde S$ is the two-sheets lifting of $S$. The projection of $\tilde E$ gives (\ref{ELRS}).\\

\emph{Step 3:} Any ending vertex of a segment $s$ in $S$ has the following two possibilities:
\begin{itemize}
  \item[(a):] there is a cycle in $L\cup R$ passing through it.
  \item[(b):] there is no any cycle in $L\cup R$ passing through it.
\end{itemize}
There is a case that an ending vertex of a segment $s_1$ in $S$ is an inner vertex of another segment $s_2$.
\begin{exm}
Suppose $s_1,s_2\in S$ are two maximal segments. One ending vertex of $s_1$ is an inner vertex of $s_2$. We can represent that
$$s_1=v_1\cdots v_{k-1}v_kv_{k-1}\cdots v_1, \quad\mbox{and}\quad s_2=v_m\cdots v_k\cdots v_{l-1}v_lv_{l-1}\cdots v_k\cdots v_m.$$
We can construct another two segments $s_3$ and $s_4$ by
$$s_3=v_1\cdots v_k\cdots v_m\cdots v_k\cdots v_1, \quad\mbox{and}\quad s_4=v_k\cdots v_{l-1}v_lv_{l-1}\cdots v_k.$$
The relation is $s_1\cup s_2=s_3\cup s_4$.
\end{exm}

We introduce the concept ``fork" to describe this situation. The set of forks $F$ is constructed form $S$ in the following way.

Let $F=S$ first, that is any segment in $S$ can be considered as a fork in $F$. If any two forks $f_1$ and $f_2$ in $F$ satisfy that one ending vertex of $f_1$ is an inner vertex of $f_2$, then we connect $f_1$ and $f_2$ to form a new fork in $F$, and replace the forks $f_1$ and $f_2$ by the new fork. Continue this process until all the forks in $F$ are maximal.  That is, any two maximal fork in $F$ are disconnected at the ending vertex, namely, any ending vertex of a maximal fork is not an inner vertex of another fork. At last, we get a set of maximal forks, still denoted by $F$.
\begin{defn}
Denote the set of edges of a fork $f$ by $e_f$. The bifurcation degree of a bifurcation vertex $v$ is the degree of $v$ in $G[e_f]$, which is the subgraph induced by $e_f$.
\end{defn}

Noticing that any two maximal forks $f_1$ and $f_2$ in $F$ only having common inner vertices are not connected to form a larger fork, we have that
\begin{prop}
Suppose $v$ is a bifurcation vertex of $f$. The bifurcation degree of $v$ is three.
\end{prop}
\begin{proof}
Let $v$ be a bifurcation vertex of $f$. If the bifurcation degree is more than three, there must be more than one segments having ending vertex at $v$, which contradicts to the construction of $S$.
\end{proof}

Now we obtain from $E=L\cup R\cup D\cup S$ that
$$E=L\cup R\cup D\cup F,$$
where $F$ is the set of maximal forks.
The lifting of $F$ into $\tilde G$ is denoted by $\tilde F$, with $p(\tilde f_i)=f_i$ for any $\tilde f_i\in\tilde F$. Suppose $f_i$ is a maximal fork in $F$. $f_i=\cup_j s_{i_j}$, where $s_{i_j}$ are maximal segments contained in $f_i$. $\tilde f_i$ is a closed eulerian path consists of all the cycles $\tilde s_{i_j}$. It is easy to see that $\tilde F$ is a set of circuits, and any circuit $\tilde f$ in $\tilde F$ must contain $2k$ numbers of auxiliary edges with $k\geq1$. Therefore,
$$\tilde E=\tilde L\dot\cup \tilde R\dot\cup \tilde D\dot\cup\tilde F,$$
where $\tilde F$ is the two-sheets lifting of $F$.\\

\emph{Step 4:} It is obvious that each element in $F=E-L-R-D$ is closed, namely, any fork $f$ in $F$ must admit a walk returning back to its starting vertex. A fork $f_i$ may contain some cycles in $E$, that is, the edge-set of $f_i$ contains some cycles. If so, the covering of each cycle in $F$ is double. We pick out one sheet for each cycle to form a set $H_{1i}$ by
\begin{eqnarray*}
H_{1i}=\{h_{i_j}=v_{j_1}v_{j_2}\cdots v_{j_k}v_{j_1}| v_{j_l}\in f_i, j_1\neq j_2\neq\cdots\neq j_k, k\geq3\}.
\end{eqnarray*}
The remaining edges in each $f_i$ still form a closed walk. To see that, we need to lifting each cycle $h_{i_j}$ into $\tilde G$.

Suppose $h_{i_j}$ is a cycle in a fork $f_i$. The two-sheets lifting of $f_i$ is denoted by $\tilde f_i$, which is symmetric in $\tilde G$, that is, the subgraph $\tilde f_i\cap G_1$ is isomorphic to the subgraph $\tilde f_i\cap G_2$, by the lifting of $F$. $f_i$ provides a two-sheets covering of $h_{i_j}$, hence there is no any cycle in $\tilde L\dot\cup\tilde R\dot\cup\tilde D$ whose projection has a common edge with $h_{i_j}$. Suppose $h_{i_j}$ can be expressed by
$$h_{i_j}=v_{j_1}v_{j_2}\cdots v_{j_k}v_{j_1}.$$
There are two cycles in $\tilde G$ contained in $\tilde f_i$, that are isomorphic to $h_{i_j}$, respectively, namely,
$$h^1_{i_j}=v^1_{j_1}v^1_{j_2}\cdots v^1_{j_k}v^1_{j_1},\quad \mbox{and}\quad h^2_{i_j}=v^2_{j_1}v^2_{j_2}\cdots v^2_{j_k}v^2_{j_1}.$$
We define the lifting of $h_{i_j}$ to be any one of $\{h^1_{i_j},h^2_{i_j}\}$, e.g. $h^1_{i_j}$. Then the lifting of $H_{1i}$ can be defined by
$$\tilde H_{1i}=\{h^1_{i_j}=v^1_{j_1}v^1_{j_2}\cdots v^1_{j_k}v^1_{j_1}| v^1_{j_l}\in \tilde f_i\cap G_1, j_1\neq j_2\neq\cdots\neq j_k, k\geq3\}.$$

All forks $f_i\in F$ provide a union set $H_1=\cup_i H_{1i}$, which is a set of cycles in $E$. The lifting of $H_1$ is denoted by $\tilde H_1$, which is the union of $\tilde H_{1i}$, i.e., $\tilde H_1=\cup_i \tilde H_{1i}$. It is easy to see from the definition of $\tilde h_{i_j}$ that each cycle in $\tilde H_1$ is a one-sheet lifting of a cycle in $H_1$, and the set $\tilde H_1$ is the one-sheet lifting of $H_1$.

\begin{prop}\label{prop4.7}
For an arbitrary fork $f\in F$, the complement of some cycles in $f$ is a closed walk.
\end{prop}
\begin{proof}
Suppose $f$ is a fork in $F$. Since $f$ is a union of some segments, the two-sheets lifting $\tilde f$ of $f$ in $\tilde E$ is a union of some cycles, hence a circuit, which is denoted by $\tilde f$. Denote the cycles contained in $f$ by $h_1, h_2, \cdots, h_j$, whose liftings are denoted by $\tilde h_1, \tilde h_2, \cdots,\tilde h_j$. Since $\tilde h_1, \tilde h_2, \cdots,\tilde h_j$ are all cycles in $\tilde G$, and the lifting $\tilde f$ contains the auxiliary edge between any pair of $(v^1_{\alpha},v^2_{\alpha})$ in $\tilde f$ lifted from an ending vertex or a bifurcation vertex, by the same argument in Proposition \ref{prop4.4}, one obtains that $\tilde f-\cup_{i=1}^j\tilde h_i$ is still a closed eulerian path in $\tilde G$, hence is a circuit. Therefore, the eulerian path of $\tilde f-\cup_{i=1}^j\tilde h_i$ gives a closed walk of $p(\tilde f-\cup_{i=1}^j\tilde h_i)$.
\end{proof}
For each fork $f_i\in F$, we denote the complement of the union of all the cycles $h_{i_j}$ in $f_i$ by $b_i$, i.e., $b_i=f_i-\cup_j h_{i_j}$. We see from Proposition \ref{prop4.7} that $b_i$ is a closed walk. Let $B$ denote the set of all the closed walk $b_i$ constructed from $f_i\in F$. It is obvious that $B$ is one-to-one corresponding to $F$. The relation is $F=B\cup H_1$.

\begin{defn}
Elements in $B$ are called the branches. A branch is a closed walk, whose set of edges are passed by the branch once or twice. Moreover, the subgraph induced by the edges, which are covered only once by the branch, is a union of cycles.
\end{defn}
%

There are two kinds of edges in $B$. Let $B^O$ be the set of edges in $B$, which are covered only once, and let $B^T$ be the set of edges in $B$, that are covered exactly twice. Denoting the edge-sets of $B$, $B^O$ and $B^T$ by $e(B)$, $e(B^O)$ and $e(B^T)$, respectively, we have $e(B)=e(B^O)\dot\cup e(B^T)$. For any branch $b_i\in B$, the edge-set of $b_i$ is denoted by $e(b_i)$. We denote that $e(b^O_i)=e(b^i)\cap e(B^O)$ and $e(b^T_i)=e(b_i)\cap e(B^T)$. Then we have the following description of a branch.
\begin{prop}
Denote the graph induced by the edge-set $e(b_i)$ by $G[e(b_i)]$. Each edge in $e(b^T_i)$ is a cut edge of $G[e(b_i)]$, and any edge in $e(b^O_i)$ is not a cut edge of $G[e(b_i)]$.
\end{prop}

\begin{proof}
It is obvious that any edge in $e(b^O_i)$ is not a cut edge of $G[e(b_i)]$. We only need to prove the former assertion. If not, there is a non-cut edge $e$ in $e(b^T_i)$, which in located on a fork in $b_i$. Then $e$ is a part of a cycle $C$ in $b_i$. Such a cycle $C$ is also a cycle in the corresponding fork $f_i$, hence must be in $H_i$. So the cover of $e$ is at most one-sheet in $b_i$. It is a contradiction to the assumption that $e$ is an edge in $e(b^T_i)$.
\end{proof}

This proposition provides another description of a branch $b\in B$.
\begin{prop}\label{prop4.10}
A branch $b\in B$ admits a closed walk on a connected subgraph of $G$, which pass the edges on a circuit once and pass all the cut edges twice.
\end{prop}\qed

So a branch consists of some circuits and some forks. Circuits in a branch is connected by some forks and there are some other forks that connect to a circuit and do not connect any other circuit in the branch. Visually speaking, the forks connected two circuits in a branch like some ``bridges", and the other forks are ``growing" from circuits. We call a maximal fork in a branch the \emph{branch fork}. A branch fork is a double cover of connected cut edges in the branch. The ending vertices of a branch fork are called the \emph{$f_b$-ending vertices}. An $f_b$-ending vertex of a branch fork is either connected to a circuit in the branch, or is an ending vertex of the branch.

We denote the lifting of a branch $b_i$ by $\tilde b_i$, which is defined to be the complement of $\tilde H_i$ in $\tilde f_i$, namely, $\tilde b_i=\tilde f_i-\cup_j\tilde h_{i_j}$. $\tilde b_i$ is a circuit in $\tilde G$. A fork in $b_i$ has a double cover by $\tilde b_i$, so its lifting is a two-sheets lifting in $\tilde b_i$. A circuit in $b_i$ has a single cover by $\tilde b_i$, so its lifting in $\tilde b_i$ is a one-sheet lifting. Therefore, for a branch $b_i\in B$, we call $\tilde b_i$ is its \emph{mixed-sheet lifting}. The union of such circuits $\tilde b_i$ gives a \emph{mixed-sheet lifting set} of $B$, denoted by $\tilde B$.

So we have the decomposition
\begin{eqnarray}\label{equ4.2}
E=L\cup R\cup D\cup H_1\cup B,
\end{eqnarray}
where $L$, $R$, $D$, $H_1$ can be regarded as sets of cycles and $B$ is a set of branches. Moreover,
\begin{eqnarray}\label{equ4.3}
\tilde E=\tilde L\dot\cup \tilde R\dot\cup\tilde D\dot\cup \tilde H_1\dot\cup \tilde B,
\end{eqnarray}
where $\tilde L$, $\tilde R$, $\tilde D$, $\tilde H_1$, $\tilde B$ are all sets of circuits in $\tilde G$, and $\tilde L$, $\tilde R$, $\tilde D$, $\tilde H_1$ are the one-sheet liftings of $L$, $R$, $D$, $H_1$, respectively, while $\tilde B$ is a mixed-sheet lifting of $B$.

If $\tilde B$ is a one-sheet lifting of $B$, we finish the proof by providing the double cycle cover $\tilde E=\tilde L\dot\cup \tilde R\dot\cup\tilde D\dot\cup \tilde H_1\dot\cup \tilde B$.\\

\emph{Step 5:} Generally speaking, $\tilde B$ may not be a one-sheet lifting of $B$. Before we reduce the mixed-sheet lifting set $B$ to some one-sheet lifting set, we turn to the cycles in $L\cup R\cup D\cup H_1$.

Consider the set of edges in $L\cup R\cup D\cup H_1$, denoted by $e(LRDH)$. The graph $G[e(LRDH)]$ induced by $e(LRDH)$ can be separated into several disconnected subgraphs. That is,
$$G[e(LRDH)]=\dot\cup_k G^{LRDH}_{k},$$
where $G^{LRDH}_{k}$ is a disconnected subgraph of $G[e(LRDH)]$ for each $k$. We now connect some cycles together in $L\cup R\cup D\cup H_1$, whose edges are in the same connected induced subgraph. Namely, any cycles $r_1,r_2,\cdots,r_i$ are connected if there is an index $k$ such that $r_j\subseteq G^{LRDH}_{k}$ for any $j\in\{1,2\cdots,i\}$, and there is no any other $r_p\in L\cup R\cup D\cup H_1$, $p\notin \{1,2,\cdots,i\}$, satisfies this property. Thus, we separate the set of cycles $L\cup R\cup D\cup H_1$ into several disconnected subsets,
$$L\cup R\cup D\cup H_1=\cup_i ((L\cup R\cup D\cup H_1)\cap G^{LRDH}_{i})=:\cup_i (L\cup R\cup D\cup H_1)_i.$$
More precisely, we denote the cycles in $L\cup R\cup D\cup H_1$ by $r_1,r_2,\cdots,r_m$. They are separated to several classes, each of which is a set of cycles in the same connected induced subgraph. That is,
\begin{eqnarray*}
\begin{split}
\{r_1,r_2,\cdots,r_m\}&=\{r_{1_1},r_{1_2},\cdots,r_{1_{p_1}}\}\cup\cdots\cup\{r_{k_1},r_{k_2},\cdots,r_{k_{p_k}}\}\\
&=\cup_i\{r_{i_1},r_{i_2},\cdots,r_{i_{p_i}}\},
\end{split}
\end{eqnarray*}
where $p_i$ is an index with $p_i\geq 1$ for each $1\leq i\leq k$, and $k$ is an index satisfying $1\leq k\leq m$. For each $i$, the set of cycles given by $$\{r_{i_1},r_{i_2},\cdots,r_{i_{p_i}}\}=(L\cup R\cup D\cup H_1)\cap G^{LRDH}_{i}=:(L\cup R\cup D\cup H_1)_i,$$
is called a \emph{connected class}.
The one-sheet lifting set of a connected class $\{r_{i_1},r_{i_2},\cdots,r_{i_{p_i}}\}$ is defined by the union of the one-sheet liftings of its elements $r_{i_1},r_{i_2},\cdots,r_{i_{p_i}}$. Precisely, denoting the one-sheet lifting in $\tilde L\dot\cup \tilde R\dot\cup\tilde D\dot\cup \tilde H_1$ of each $r_{i_j}$ by $\tilde r_{i_j}$, and denoting the lifting set of each connected class $(L\cup R\cup D\cup H_1)_i$ by $(\tilde L\dot\cup \tilde R\dot\cup \tilde D\dot\cup\tilde H_1)_i$, we can get
$$(\tilde L\dot\cup \tilde R\dot\cup \tilde D\dot\cup \tilde H_1)_i=\{\tilde r_{i_1},\tilde r_{i_2},\cdots,\tilde r_{i_{p_i}}\},$$
and
$$\tilde L\dot\cup \tilde R\dot\cup \tilde D\dot\cup \tilde H_1=\cup_i(\tilde L\dot\cup \tilde R\dot\cup \tilde D\dot\cup \tilde H_1)_i.$$
So we get the decomposition that
\begin{eqnarray}\label{equ4.4}
\tilde E=(\cup_i(\tilde L\dot\cup \tilde R\dot\cup \tilde D\dot\cup \tilde H_1)_i)\dot\cup \tilde B,
\end{eqnarray}
and its projection
\begin{eqnarray}\label{equ4.5}
E=(\cup_i(L\cup R \cup D\cup H_1)_i)\cup B.
\end{eqnarray}
(\ref{equ4.4}) and (\ref{equ4.5}) are rearrangements of (\ref{equ4.3}) and (\ref{equ4.2}), respectively. One can find that the lifting set $(\tilde L\dot\cup \tilde R\dot\cup \tilde D\dot\cup \tilde H_1)_i$ of a connected class $(L\cup R\cup D\cup H_1)_i$ may not be connected in $\tilde G$. A connected class $(L\cup R\cup D\cup H_1)_i$ can cover some edges in its edge-set (or the induced graph $G^{LRDH}_{i}$ of its edge-set) twice.

Let's consider the connection of any two different cycles in $(\tilde L\dot\cup \tilde R\dot\cup \tilde D\dot\cup \tilde H_1)_i$. Suppose $r_1$ and $r_2$ are two cycles in $(\tilde L\dot\cup \tilde R\dot\cup \tilde D\dot\cup \tilde H_1)_i$, and $r_1\neq r_2$. Denote the union of $r_1$ and $r_2$ by $r=r_1\cup r_2$. Let $P=r_1\cap r_2$ be the graph of the intersection of $r_1$ and $r_2$. $P$ is a set of common edges and vertices.

There are a pair of vertices $(v_1,v_2)$ such that $v_1,v_2\in P$ and there are two paths $p_1\in r_1-P$, $p_2\in r_2-P$ connected $v_1$ and $v_2$, respectively. Such pair of vertices exist when $r_1$ and $r_2$ are not coincident. Such vertices pair $(v_1,v_2)$ may not be unique. Moreover, $v_1$ might be the same vertex of $v_2$, which implies $v_1=v_2$ is the only common part, i.e., $P=\{v=v_1=v_2\}$.

Let $r'_1=r_1-p_1$, $r'_2=r_2-p_2$ be two paths in cycles $r_1$ and $r_2$, respectively. Let $Q=r'_1\cup r'_2$ be the graph containing the common part $P=r_1\cap r_2$. $P=Q=\{v\}$, when $v_1=v_2=v$. Any two paths from $v_1$ to $v_2$ in $Q$ may share some common vertices. Hence $Q$ is a union of cycles and segments.
For any vertices $a\in p_1$ and $b\in p_2$, we have
\begin{prop}\label{prop4.11}
For any vertices $a\in p_1$ and $b\in p_2$, there exist two paths $l_1$ and $l_2$ from $a$ to $b$ such that $l_1\cup l_2=r_1\cup r_2$.
\end{prop}
\begin{proof}
Vertices $v_1,v_2$ and $a$ separate the cycle $r_1$ into three paths. Namely, the path in $p_1$ from $a$ to $v_1$ denoted by $p_1(a, v_1)$, the path in $Q\cap r_1$ from $v_1$ to $v_2$ denoted by $Q_1(v_1,v_2)$, and the path in $p_1$ from $v_2$ to $a$ denoted by $p_1(v_2,a)$. Analogously, the three vertices $v_1,v_2$ and $b$ separate $r_2$ into three paths. We denote them by $p_2(b,v_1)$ the path in $p_2$ from $b$ to $v_1$, $Q_2(v_1,v_2)$ the path in $Q\cap r_2$ from $v_1$ to $v_2$, and $p_2(v_2,b)$ the path in $r_2$ from $v_2$ to $b$. Then the two paths $l_1$ and $l_2$ can be given by
$$l_1=p_1(a, v_1)\cup Q_1(v_1,v_2)\cup p_2(v_2,b),\quad l_2=p_2(b,v_1)\cup Q_2(v_1,v_2)\cup p_1(v_2,a).$$
It is obvious that
\begin{eqnarray*}
\begin{split}
l_1\cup l_2&=p_1(a, v_1)\cup Q_1(v_1,v_2)\cup p_2(v_2,b)\cup p_2(b,v_1)\cup Q_2(v_1,v_2)\cup p_1(v_2,a)\\
&=(p_1(a, v_1)\cup Q_1(v_1,v_2)\cup p_1(v_2,a))\cup (p_2(b,v_1)\cup Q_2(v_1,v_2)\cup p_2(v_2,b))\\
&=r_1\cup r_2.
\end{split}
\end{eqnarray*}
\end{proof}
For any vertices $a\in p_1$ and $b\in r_2-P\cap r_2-p_2$, we have
\begin{prop}\label{prop4.12}
For any vertices $a\in p_1$ and $b\in r_2-P\cap r_2-p_2$, there exist two paths $l_1$ and $l_2$ from $a$ to $b$ such that $r_1\cup r_2-l_1\cup l_2$ is a cycle.
\end{prop}
\begin{proof}
As in the proof of Proposition \ref{prop4.11}, vertices $v_1,v_2$ and $a$ separates the cycle $r_1$ into three paths $p_1(a, v_1)$, $Q_1(v_1,v_2)$, and $p_1(v_2,a)$. The three vertices $v_1,v_2$ and $b$ separates $r_2$ into three paths. We denote them by $Q_2(b,v_2)$ the path in $Q\cap r_2$ from $b$ to $v_2$, $p_2$ the path from $v_2$ to $v_1$ in $p_2$, and $Q_2(v_1,b)$ the path in $Q\cap r_2$ from $v_1$ to $b$. Then the two paths $l_1$ and $l_2$ can be given by
$$l_1=p_1(a, v_1)\cup Q_2(v_1,b),\quad l_2=Q_2(b,v_2)\cup p_1(v_2,a).$$
It is obvious that
\begin{eqnarray*}
\begin{split}
r_1\cup r_2&=p_1(a, v_1)\cup Q_1(v_1,v_2)\cup p_1(v_2,a)\cup Q_2(b,v_2)\cup p_2\cup Q_2(v_1,b)\\
&=(p_1(a, v_1)\cup Q_2(v_1,b))\cup (Q_2(b,v_2)\cup p_1(v_2,a))\cup (Q_1(v_1,v_2)\cup p_2)\\
&=l_1\cup l_2\cup (Q_1(v_1,v_2)\cup p_2).
\end{split}
\end{eqnarray*}
Equivalently,
$$r_1\cup r_2-l_1\cup l_2=(Q_1(v_1,v_2)\cup p_2),$$
which is a cycle, since $p_2$ and $Q$ have no common edge.
\end{proof}
It follows directly from Propositions \ref{prop4.11} and \ref{prop4.12} that
\begin{prop}\label{prop4.13}
Suppose $r_1$ and $r_2$ are two distinct cycles connected to each other. The common subgraph of $r_1,r_2$ is denoted by $P=r_1\cap r_2$. For any two vertices $a\in r_1-P$ and $b\in r_2-P$. There is two paths $l_1,l_2$ from $a$ to $b$ such that
$$l_1\cup l_2=r_1\cup r_2-\delta,$$
where $\delta$ is a cycle or an empty graph.
\end{prop}\qed\\

\emph{Step 6:} Obviously, all the ending vertices of a fork $f_i\in F$ are contained in the corresponding branch $b_i=f_i-\cup_jh_{i_j}\in B$. According to the definitions, all the branches $\{b_i\}$ are disconnected at their ending vertices mutually. That is, any ending vertex of a branch $b_1$ is not any vertex of another branch $b_2$.

The projection $E$ of $\tilde E$ is a double cover of $G$, and the degree of each vertex in $G$ is at least 2. Hence, each ending vertex $u_{i_l}$ of a branch $b_i\in B$ must be passed by some cycles in $L\cup R\cup D\cup H_1$. Equivalently, any ending vertex of a branch $b\in B$ must connect to a cycle from one connected class. Noticing that a branch $b\in B$ is maximal and any two connected class are disconnected, one can assert that any $f_b$-ending vertex of a branch fork is connected to a cycle, and any two $f_b$-ending vertices are connected to each other by the union of some segments and circuits.

Denote the set of branch forks by $BF$. Suppose a branch $b_i\in B$ contains some forks in $BF$, which form a set denoted by $b_i\cap BF$. Noticing that the complement of $b_i\cap BF$ in $b_i$ is some disconnected circuits, which can be decomposed into some cycles. Such cycles form a set denoted by $H_{2i}$. When we run over all the branch $b_i\in B$, we get a set of cycles $H_2=\cup_i H_{2i}$. The complement of $H_2$ in $B$ is the set of branch forks $BF$. We now have
\begin{eqnarray}\label{equ4.6}
E=(\cup_i(L\cup R\cup B\cup H_1)_i)\cup H_2\cup BF,
\end{eqnarray}
where $H_2$ is the set of cycles contained in $B$ and $BF$ is the set of branch forks. One may notice that the set $H_2$ is the copy of $H_1$. By the closeness of $BF$, we have the decomposition of $\tilde E$ that
\begin{eqnarray}\label{equ4.7}
\tilde E=(\cup_i(\tilde L\dot\cup \tilde R\dot\cup \tilde D\dot\cup \tilde H_1)_i)\dot\cup \tilde H_2\dot\cup\tilde{BF},
\end{eqnarray}
where $\tilde H_2$ is a one-sheet lifting of $H_2$ in $G_2$, and $\tilde{BF}$ is a set of circuits in $\tilde G$. However, $H_2$ and $H_1$ are two sets of cycles in $G$, which form the double cycle cover of their induced graphs $G[e(H_1)]=G[e(H_2)]$.

Setting $H=H_1\cup H_2$ and $\tilde H=\tilde H_1\dot\cup \tilde H_2$, we can represent (\ref{equ4.6}) and (\ref{equ4.7}) by
\begin{align*}
E&=(\cup_i(L\cup R\cup B\cup H)_i)\cup BF,\\
\tilde E&=(\cup_i(\tilde L\dot\cup \tilde R\dot\cup \tilde D\dot\cup \tilde H)_i)\dot\cup\tilde{BF}.
\end{align*}

Since a fork is a union of some maximal segments by connecting one ending vertex of a segment to an inner vertex of another segment, the set $BF$ contains finitely many segments. We consider the set $BF$ as a set of segments now, and denote the number of maximal segments in $BF$ by $n$. The cycle double cover conjecture has a positive answer if $BF=\emptyset$, or $n=0$. A maximal segment in $BF$ must be one of the following cases:
\begin{itemize}
  \item[(A):] neither of its two ending vertices is an inner vertex of another segment.
  \item[(B):] at least one of its ending vertices is an inner vertex of another segment.
\end{itemize}
\begin{prop}\label{prop4.14}
$BF$ must contain some maximal segments of type (A), unless $BF=\emptyset$.
\end{prop}
\begin{proof}
If not, $BF\neq\emptyset$ only contains maximal segments of type (B). Consider an arbitrary maximal segment $s_1\in BF$, one of whose ending vertices is an inner vertex of another maximal segment $s_2\in BF$. One of the ending vertices of the maximal segment $s_2$ is an inner vertex of another maximal segment $s_3\in BF$. Since the number of segments in $BF$ is finite, we can get a sequence of mutually different maximal segments $s_1,s_2,\cdots,s_k\in BF$ with $k$ being the largest index, such that one of the ending vertices of $s_i$ is an inner vertex of $s_{i+1}$. Since $s_k$ is also type (B), one of its ending vertices is an inner vertex of one maximal segment in $\{s_1,s_2,\cdots,s_{k-1}\}$, e.g $s_j$. Otherwise we can extend the sequence to be $s_1,s_2,\cdots,s_{k+1}\in BF$, which contradicts to the assumption that $k$ is the largest index. Now we have find a cycle in $s_j\cup s_{j+1}\cup\cdots\cup s_k$ whose edges are all covered twice in a branch, which is impossible according to Proposition \ref{prop4.10}.
\end{proof}

Let's consider the connection of a segment $s\in BF$ and a cycle $r$, when the intersection of $s$ and $r$ is only a vertex $v$. Such a cycle is chosen from $\cup_i(L\dot\cup R\dot\cup D\dot\cup H)_i$ for some $i$, and $v$ is one of the ending vertex of $s$ connected to $r$. So we have
\begin{prop}\label{prop4.15}
Suppose a maximal segment $s$ and a cycle $r$ are connected at a vertex $v$. There are two paths $l_1$ and $l_2$ from the ending vertex $a\neq v$ of $s$ to an arbitrary vertex $b\neq v$ in $r$, such that $l_1\cup l_2=s\cup r$.
\end{prop}
\begin{proof}
The segment $s$ gives two paths $p_1=p_2$ from $a$ to $v$. Vertices $v$ and $b$ separate the cycle $r$ into two paths $p_3$, $p_4$. So we can set $l_1=p_1\cup p_3$ and $l_2=p_2\cup p_4$ to finish the proof.
\end{proof}

We pick out one maximal segment $s$ of type (A) in $BF$. The ending vertices of $s$ are called $\alpha$, $\beta$. The segment $s$ provides two paths $p_1=p_2$ from $\alpha$ to $\beta$. We have
\begin{prop}
There are another two paths $p_3,p_4$ from $\alpha$ to $\beta$ in $E$, such that $p_m$ has no any common edge with $p_1=p_2$, for $m=3,4$.
\end{prop}
\begin{proof}
Firstly, we claim that there must be another path $p$ from $\alpha$ to $\beta$ in $E$ besides $p_1$ and $p_2$. If not, $p_1=p_2$ is the only path from $\alpha$ to $\beta$ in $E$, which means that the induced graph of the segment $s$ is a bridge in $E$.

Moreover, $p$ has no any common edge with $p_1=p_2$ since the segment $s$ is the double cover of the path $p_1=p_2$. $p$ must pass some maximal segments $\{s_j\}_{j=1}^a$  (chosen from some branch forks) and cycles $\{r_i\}_{i=1}^b$ (chosen from $\cup_i(L\cup R\cup D\cup H)_i$), where the segments are in $BF-s$ and the cycles are in $E-BF$. If an edge is covered by two different cycles, we only choose one of them up to $\{r_i\}_{i=1}^b$. More precisely, each edge in $p$ is only covered by $\{r_i\}_{i=1}^b$ once, if the edge is not covered by some segments in $\{s_j\}_{j=1}^a$. It always can be done since $E$ is connected and is a double cover of $G$. We can connect them to get a closed walk $P(\alpha,\beta)=(\cup_i r_i)\cup (\cup_j s_j)$. $p$ is only cover a connected part of $P(\alpha,\beta)$ once. By applying Propositions \ref{prop4.13} and \ref{prop4.15} inductively, we can get a sequence of cycles $r^c_i$, and two paths $p_3,p_4$ from $\alpha$ to $\beta$, such that $p_3\cup p_4\cup (\cup_{i}r^c_i)=P(\alpha,\beta)$. $p_m$ $(m=3,4)$ has no common edge with $p_1=p_2$ because the union of paths $P(\alpha,\beta)\cup p_1\cup p_2$ in $E$ at most can cover an edge in $G$ twice.
\end{proof}
Therefore, we can get two distinct circuits
$$c_1=p_1\cup p_3,\quad c_2=p_2\cup p_4,$$
which can be decomposed into some cycles.
We replace the segments $s$, $\{s_j\}_{j=1}^a$ and cycles $\{r_i\}_{i=1}^b$ by cycles in $c_1$ and $c_2$ in (\ref{equ4.6}) to get a new decomposition of $E$. It is obvious that the number of branch forks strictly decline. The set of cycles in $\cup_i(L\cup R\cup D\cup H)_i$ has been changed. The connected class of cycles may need to be rearranged. The point are that the number of segments $BF$ is reduced, and the structure of forks and connected class are preserved. Some of the maximal segments of type (A) have been replaced, however, some maximal segments of type (B) turn to type (A) in the process because the segments at their ending vertices become cycles. By Proposition \ref{prop4.14}, there always exist maximal segments of type (A) in $BF$ before all the maximal segments have been replaced by some cycles. We can continue this process until $BF=\emptyset$.

\qed\\

\section*{Acknowledgements}
The author want express great appreciate to Doctor Chenli Shen for his passionate introduction on this conjecture.
This work is supported by the Natural Science Foundation of Jiangsu Province (No. BK20160661) and supported in part by NSFC (No. 11371386).

Bin Shen \\
Department of Mathematics, Southeast University, 211189
Nanjing, P. R. China\\
E-mails: shenbin@seu.edu.cn\\

\end{document}